\documentclass{amsart}
\usepackage[a4paper,top=3cm,bottom=2cm,left=3cm,right=3cm,marginparwidth=1.75cm]{geometry}

\usepackage[]{algorithm2e} 
\usepackage{multirow}
\usepackage{graphicx} 
\usepackage{amsmath,amssymb,amsthm}
\usepackage{float}
\usepackage{todonotes}
\usepackage[english]{babel}
\usepackage[T1]{fontenc}
\usepackage{cite}
\usepackage{mathtools}
\usepackage{commath}
\usepackage{xypic}
\usepackage{float}
\usepackage{graphicx}
\usepackage{subcaption}
\usepackage{tikz} 
\usepackage{bibentry}
\usepackage{makecell}

\usepackage[colorlinks=true, allcolors=blue]{hyperref}
\usepackage{skull}
\usepackage{blindtext}
\usepackage{comment}
\usepackage[toc,page]{appendix}
\usetikzlibrary{3d}

\usepackage{listings}
\usepackage{color}

\definecolor{dkgreen}{rgb}{0,0.6,0}
\definecolor{gray}{rgb}{0.5,0.5,0.5}
\definecolor{mauve}{rgb}{0.58,0,0.82}

\newtheorem{theorem}{Theorem}
\newtheorem{lemma}[theorem]{Lemma}
\newtheorem{proposition}[theorem]{Proposition}

\theoremstyle{definition}

\newtheorem{remark}[theorem]{Remark}

\newtheorem{example}[theorem]{Example}

\newtheorem*{theorem*}{Theorem}
\newtheorem*{remark*}{Remark}
\newtheorem*{remarks*}{Remarks}
\newtheorem*{definition*}{Definition}

\newtheoremstyle{named}{}{}{\itshape}{}{\bfseries}{.}{.5em}{\thmnote{#3}#1}
\theoremstyle{named}

\newcommand{\FF}{\mathbb{F}}      
\newcommand{\ZZ}{\mathbb{Z}}     
\newcommand{\QQ}{\mathbb{Q}}      

\newcommand{\PGL}{\mathrm{PGL}}

\newcommand{\Gal}{\mathrm{Gal}}

\DeclareMathOperator{\Tr}{Tr}
\DeclareMathOperator{\N}{N}
\DeclareMathOperator{\Res}{Res}

\DeclareMathOperator{\Aut}{Aut}

\DeclareMathOperator{\Disc}{Disc}

\title{Dynamical Irreducibility of Certain Families of Polynomials over Finite Fields}
\author[Day]{Tori Day}
\address{Department of Mathematics and Statistics, Mount Holyoke College, South Hadley, MA 01075 USA}
\email{tori.day@mtholyoke.edu}
\author[Deland]{Rebecca DeLand}
\address{Department of Mathematics, University of Colorado, Boulder, CO 80309 USA}
\email{rebecca.deland@colorado.edu}
\author[Juul]{Jamie Juul}
\address{Department of Mathematics, Colorado State University, Fort Collins, CO 80523 USA}
\email{jamie.juul@colostate.edu}
\author[Thomas]{Cigole Thomas}
\address{Department of Mathematics, Colorado State University, Fort Collins, CO 80523 USA}
\email{cigole.thomas@colostate.edu}
\author[Thompson]{Bianca Thompson} 
\address{Department of Mathematics, Westminster University, Salt Lake City, UT 84105 USA}
\email{bthompson@westminsteru.edu}
\author[Tobin]{Bella Tobin}
\address{Department of Mathematics, Agnes Scott College, Decatur, GA 30030 USA}
\email{btobin@agnesscott.edu}
\date{}

\keywords{dynamical irreducibility, stable}

\begin{document}

\begin{abstract}
 We determine necessary and sufficient conditions for unicritical polynomials to be dynamically irreducible over finite fields. This result extends the results of Boston-Jones and Hamblen-Jones-Madhu regarding the dynamical irreducibility of particular families of unicritical polynomials. We also investigate dynamical irreducibility conditions for cubic and shifted linearized polynomials. 
\end{abstract}
\maketitle

\section{Introduction}
Let $K$ be a field and $f(x) \in K[x]$ a polynomial with degree at least $2$. We will consider the reducibility of \[f^n(x) \vcentcolon=(\underbrace{f\circ f\circ \ldots \circ f}_{n\ \text{times}})(x).\] We say a polynomial $f$ is \emph{dynamically irreducible} (or \textit{stable}) if $f^n$ is irreducible for all positive integers $n$ over $K$. Similarly, $f$ is \emph{eventually dynamically irreducible} (or \textit{eventually stable}) if there exists positive integers $r, N$ such that $f^{N+n}=g_1(f^n)\dots g_r(f^n)$ and $g_i(f^n)$ is irreducible for all non-negative integers $n$. In other words, $f$ is eventually dynamically irreducible if the number of irreducible factors of $f^n$ stabilizes. For $a\in K$ we call $(f,a)$ a \textit{dynamically irreducible pair} if $f^n(x)-a$ is irreducible for all positive integers $n$. Note that the dynamical irreducibility of $f$ is equivalent to the dynamical irreducibility of the pair $(f,0)$. See \cite[Section 19]{trends} for a brief overview of dynamical irreducibility.

One motivation for studying eventual stability comes from arboreal Galois representations. For a rational map $f(x)\in K(x)$ and an element $a\in K$, the \textit{backward orbit} of $a$ under $f$, \[\mathcal{O}_f^-(a)=\{b\in \bar{K}:f^n(b)=a \text{ for some } n\in \mathbb{Z}_{\geq0}\},\] has a natural structure as a regular $\deg f$-ary rooted tree, let $T$ denote this tree. The absolute Galois group of $K$ acts on the backward orbit in a way that preserves the tree structure, inducing a group homomorphism \[\rho:\Gal(\bar{K}/K)\rightarrow \Aut(T),\] where $\Gal(\bar{K}/K)$ denotes the absolute Galois group of $K$ and $\Aut(T)$ denotes the automorphism group of $T$. This map is called an \textit{arboreal Galois representation}. See \cite{BJ-AGR2}, \cite{BJ}, \cite{Jones_survey} for a detailed discussion. The image of this representation acts transitively on the $n$-th level of the tree (i.e. on the roots of $f^n(x)-a$) for all $n$ if and only if the pair $(f,a)$ is dynamically irreducible.

Further, if $K$ is a number field, the factorization type of a polynomial in residue fields of $K$ provides information about the action of Frobenius elements on the roots of the polynomial. Let $K$ be a number field and $f$ a polynomial of degree $d\geq 2$ over $K$. Denote by $K_n$ the splitting field of $f^n(x)-a$ over $K$. If the ideal $\mathfrak{p}$ is prime in the ring of integers, $\mathcal{O}_K$, of $K$, then the residue field, $F_{\mathfrak{p}}:=\mathcal{O}_K/\mathfrak{p}$, is finite. The image of the Frobenius class, $\rho(\mathrm{Frob}_{\mathfrak{p}})$, can be described by its image in  $\mathrm{Gal}(K_n/K)$ for each $n$.  For any irreducible polynomial $g\in K[x]$, let $\bar{g} \in F_{\mathfrak{p}}$ denote the reduced polynomial mod $\mathfrak{p}$. If $\bar{g}$ is separable and $\deg \bar{g}=\deg g$, then the cycle structure of the action of $\mathrm{Frob}_{\mathfrak{p}}$ on the roots of $g$ is given by the factorization type of $\bar{g}$ in $F_{\mathfrak{p}}[x]$. Therefore, understanding the factorization of $f^n(x)-a$ over finite fields provides information about the images of Frobenius classes \cite{BJ}. 

In \cite{BJ}, Jones and Boston determine necessary and sufficient conditions for stability in the case where $f(x)$ is a degree $2$ polynomial defined over a finite field $\FF_q$ for $q = p^s$ where $p$ is an odd prime and $s$ is a positive integer. Their work is extended to polynomials of the form $x^d+c$ in \cite{Hamblen_et_al.}. Using similar techniques, we extend their result to all unicritical polynomials in Theorem \ref{prop:unicritical}. As a special case of this theorem, we deduce the following for unicritical polynomials of prime degree.

\begin{theorem}\label{thm:maintheorem}
Let $r$ be a prime and $q$ a prime power with $q \equiv 1 \mod r$. Further,
  let $f(x)\in \mathbb{F}_q[x]$ be a degree $r$ unicritical polynomial with critical point $\gamma$ and lead coefficient $a$. Then $f$ is dynamically irreducible over $\FF_q$ if and only if \[
\left\lbrace -\frac{f(\gamma)}{a}\right\rbrace \bigcup \left\lbrace \frac{f^n(\gamma)}{a}: n\in \mathbb{Z}_{>1}\right\rbrace.
\]  contains no $r$-th powers in $\mathbb{F}_q$. 
\end{theorem}
This criteria provides an effective method of finding examples of dynamically irreducible polynomials, as shown in Example \ref{ex:unicrit}.

G\'{o}mez-P\'{e}rez,  Nicol\'{a}s, Ostafe, and Sadornil, prove a necessary, but not sufficient, condition for dynamical irreducibility of polynomials by relating the parity of the number of factors of the polynomial to whether or not its discriminant is a square \cite[Theorem 3.2]{Ostafe_et_al.}, see Theorem \ref{thm:3.2ostafe} below. Using their result and a result of Dickson \cite[Theorem 3]{Dickson} which gives necessary and sufficient conditions for irreducibility of cubic polynomials, we create the following recursive test for dynamical irreducibility at the $n$th iterate for cubic polynomials. 

\begin{proposition}\label{prop:cubic} Let $p>3$ be an odd prime and $q=p^s$.
    Let $h(x)=b_3x^3+b_1x+b_0 \in \FF_q[x]$ and $\beta_0\in \FF_q$. Let $\gamma_1, \gamma_2$ denote the critical points of $h$ and for each $n$ let $\beta_n$ be a root of $h^n(x)-\beta_0$. Then the pair $(h,\beta_0)$ is dynamically irreducible over $\FF_q$  if and only if 
\begin{enumerate}
    \item\label{cond1} $-3(h^{n+1}(\gamma_1)-\beta_0)(h^{n+1}(\gamma_2)-\beta_0)$ is a non-zero square for all $n\geq 0$, and
    \item\label{cond2} for $\mu_n\in \FF_q(\beta_n)^\times$ such that $-4\left(\frac{b_1}{b_3}\right)^3-27\left(\frac{b_0-\beta_n}{b_3}\right)^2=81\mu^2$, we have $$\frac{1}{2}\left(-\left(\frac{b_0-\beta_n}{b_3}\right)+\mu_n\sqrt{-3}\right)$$ is not a cube in $\FF_q(\sqrt{-3},\beta_n)$ for all $n\geq 0$.
\end{enumerate}
\end{proposition}

We apply these criteria for two cubic polynomials in Example \ref{ex:cubiccubes} and Example \ref{ex: cubicsquares}. Additionally, we give a family of polynomials for which there is a simple test for dynamical irreducibility using a result of Chu \cite{Chu}, which relies on Dickson's theorem, and we apply this test in Example \ref{ex:Chu}.
In general, there are few polynomials that are dynamically irreducible over a finite field. For instance, in Remark \ref{rmk:reducibility}, we see some values of $p$ for which there are no dynamically irreducible cubics of the form $h(x) = x^3+b_1x+b_0$ over $\mathbb{F}_p$. 
In Example \ref{ex:Propcubic}, we provide a polynomial that is dynamically irreducible.  

We also examine the dynamical irreducibility of shifted linearized polynomials, expanding an argument of Ahmadi et al \cite{Ahmadi_et_al} for quadratic polynomials which is also used in \cite{Ostafe_et_al.} for cubic polynomials, to show the second iterate must be reducible, unless $\operatorname{char} K = \deg f=2$, in which case the third iterate must be reducible.

Dynamical irreducibility of polynomials has been studied extensively in the quadratic case over finite fields \cite{goksel2023note, GokselXiaBoston, GomezNicolas, OstafeShparlinski} as well as over $\QQ$ \cite{Goksel, DS, Merai_et_al., demarketal}. Darwish and Sadek \cite{DS} find a family of polynomials (Dumas polynomials) which they prove, under some mild conditions, are dynamically irreducible. Further, they show that if the family of iterates of a polynomial contains a Dumas polynomial, then the polynomial must be dynamically irreducible over $\QQ.$  In \cite{Heath-BrownMicheli}, the authors study irreducibility of compositions of quadratic polynomials. 

\subsection{Outline} In Section \ref{sec:background}, we review definitions and known results related to dynamical irreduciblity that we build off of in the remainder of the paper. In Section \ref{sec:unicrit}, we prove Theorem \ref{thm:maintheorem} and establish this result as a generalization of \cite[Proposition 2.3]{BJErrata} and \cite[Theorem 8]{Hamblen_et_al.}. In Section \ref{sec:Cubic polynomials}, we summarize some known results about dynamical irreducibility of cubic polynomials along with some computational examples. Finally, in Section \ref{Sec:Artin}, we study shifted linearized polynomials and show that a shifted linear degree $p$ polynomial over a field of characteristic $p$ must have reducible second iterate for $p\neq 2$ and reducible third iterate when $p=2$.

\section{Background and Preliminaries}\label{sec:background}

Throughout the paper we will consider $f(x)$ a polynomial over a field $K$. We will primarily focus on polynomials over finite fields, where we will write $K=\mathbb{F}_q$ where $q = p^s$ for some prime $p$ and positive integer $s$. 

Conjugation preserves many dynamical properties of a map, including, as we will see in the following lemma, dynamical irreducibility in some sense. We say a polynomial $f \in K[x]$ is \textit{conjugate} to $g \in \overline{K}[x]$ if there exists some $\varphi(x) = ax+b\in \PGL_2(\overline{K})$ such that $g = f^\varphi := \varphi\circ f \circ \varphi^{-1}.$ This conjugation commutes with iteration, so $(f^\varphi)^n = (f^n)^\varphi$, and thus preserves the dynamics of corresponding points. 

\begin{lemma}\label{Lemma: Conjugation}
Let $f(x) \in K[x]$ and $\varphi(x) = cx+\alpha \in \PGL_2(K)$. Define $g(x) = f^\varphi(x)$. Then, $f^n(x)$ is irreducible over $K$ if and only if $g^n(x)-\alpha$ is irreducible over $K$. In particular, $f$ is dynamically irreducible over $K$ if and only if the pair $(g,\alpha)$ is dynamically irreducible. 
\end{lemma}
\begin{proof}
Since $g^n(x) = (f^\varphi)^n(x)=(f^n)^\varphi(x)$, then $g^n(x)-\alpha= cf^n\left(\frac{x-\alpha}{c}\right)$. Since they have the same splitting field, the polynomial $f^n(x)$ is irreducible if and only if $f^n\left(\frac{x-\alpha}{c}\right)$ is irreducible, which is clearly irreducible if and only if $g^n(x)-\alpha$ is irreducible. 
\end{proof}

We demonstrate the utility of Lemma \ref{Lemma: Conjugation} in the quadratic case. This technique will be especially useful in the proof of Theorem \ref{thm:maintheorem}, where we will conjugate unicritical polynomials to polynomials with critical point $0$.
\begin{example}
For $f(x)=ax^2+bx+c$, define $h(x)=f(x-\frac{b}{2a})+\frac{b}{2a} = ax^2-\frac{b^2}{4a}+c+\frac{b}{2a}$. To check whether $f^n(x)$ is irreducible, it suffices to check whether $h^n(x)-\frac{b}{2a}$ is irreducible.
\end{example}

Lemma \ref{Lemma: Conjugation} also implies that the irreducibility of a polynomial is preserved when the polynomials is conjugated to a monic form. Specifically, let $f(x)=a_dx^d+a_{d-1}x^{d-1}+\dots+a_0$ be a degree $d$ polynomial where $a_d$ is a $d-1^{st}$ power in $K$, say $a_d=c^{d-1}$ for $c \in K^\times$. Then $g(x)=cf(\frac{x}{c})$ is monic, and using Lemma \ref{Lemma: Conjugation} with $\alpha = 0$ we have that $f^n(x)$ is irreducible over $K$ if and only if $g^n(x)$ is irreducible over $K$. 

\begin{example}\label{ex:cubicleadcoeff}
    Let $f(x)=a_3x^3+a_2x^2+a_1x+a_0 \in \mathbb{F}_p[x]$ for some prime $p$ with $a_3 \neq 0$. If $a_3$ is a square, then for $c \in \mathbb{F}_p^\times$ such that $c^2 = a_3$, the dynamical irreducibility of $f(x)$ is equivalent to checking dynamical irreducibility of the monic polynomial $g(x)=cf(\frac{x}{c})=x^3+ \frac{a_2}{c}x^2+a_1x+ca_0$. If $a_3$ is not a square, then for any  $c \in \FF_p^\times$, the leading coefficient of $cf(\frac{x}{c})$ is $\frac{a_3}{c^2}$, which is also not a square. 
\end{example}

There are a number of classical results that can be used to determine whether a polynomial and its iterates are irreducible. A powerful tool in proving dynamical irreducibility of a polynomial is Capelli's Lemma, which gives necessary and sufficient conditions for irreducibility of a composition of polynomials. 

    \begin{lemma}[Capelli's Lemma]\label{Lemma:Capelli's} Let $K$ be a field, $f(x), g(x)\in K[x]$, and let $\beta\in\bar{K}$ be any root of $g(x)$. Then $g(f(x))$ is irreducible over $K$ if and only if both $g$ is irreducible over $K$ and $f(x)-\beta$ is irreducible over $K(\beta)$.
    \end{lemma}

In order to state other existing results we first recall some basic definitions. Let $f$ and $g$ be polynomials over a field $K$ of degrees $d$ and $k$, respectively, with leading coefficients $a_d$ and $b_k$. Let $\lambda_1,\dots,\lambda_d$ and $\mu_1,\dots,\mu_k$ denote the roots of $f$ and $g$ respectively, counted with multiplicity. Then the \textit{resultant}, Res$(f,g)$, is defined as
    \[\Res(f,g)= a_d^{k}b_k^{d}\prod_{1\leq i\leq d, 1\leq j\leq k} (\lambda_i-\mu_j)= (-1)^{dk}b_k^d\prod_{1\leq j\leq k} f(\mu_j).\]
    The \textit{discriminant} of $f$, denoted by Disc$(f)$, is defined as 
    \[\text{Disc}(f)=a^{2d-2}_d \prod_{i<j}(\lambda_i -\lambda_j)^2 = (-1)^{d(d-1)/2}a_d^{d-k-2}\Res (f, f')\]
    where $f'$ is the derivative of $f$ and $k=\deg f'$. The discriminant and resultant are elements of the field $K$.
    
Let $L/K$ be a separable extension of fields and let $\sigma_1,\dots,\sigma_r$ be the distinct embeddings of $L$ into an algebraic closure of $K$. The \textit{norm} and \textit{trace} of an element $\alpha \in K$ from $L$ to $K$ are defined, respectively, as 
\[
N_{L/K}(\alpha)=\prod_{i=1}^r\sigma_i(\alpha)\ \text{and}\ \Tr_{L/K}(\alpha)=\sum_{i=1}^r\sigma_i(\alpha).
\]
Note, if $\alpha\in K$, then $\Tr_{L/K}(\alpha) = [L:K]\cdot \alpha$. Hence, if $\operatorname{char} K \mid [L:K]$, then $\Tr_{L/K}(\alpha)=0$, this will be important to our arguments in Section \ref{Sec:Artin}. We make use of the following formulas for norms and traces in towers of extensions.

\begin{theorem}[Theorem 5.1, Chapter VI, \cite{lang02}]\label{Lemma: Transitivity of Norm and Trace}
    Let $K$ be a finite field, $F$ be a finite extension of $K$ and $E$ be a finite extension of $F$. Then, for all $\alpha \in E$, we have
    \begin{enumerate}
    \item $\N_{E/K} = \N_{F/K}\left(\N_{E/F}(\alpha)\right)$
        \item $\Tr_{E/K}(\alpha) = \Tr_{F/K}\left(\Tr_{E/F}(\alpha)\right)$
    \end{enumerate}
\end{theorem}

The norm function can be used to determine if an element of a finite field $\mathbb{F}_{q^s}$ is an $r$-th power by mapping it to its norm in $\mathbb{F}_{q}$. We state and prove this well-known relationship below. We will use this to study the dynamical irreducibility of unicritical polynomials in Section \ref{sec:unicrit}. 

\begin{lemma}\label{lem:rthroot}
    Let $\FF_{q^s}$ be a finite extension of $\FF_q$ and let $\alpha \in \FF_{q^s}$ and let $r>0$ be an integer with $\gcd(q,r)=1$. Then
    \begin{enumerate}
    \item\label{part1} if $\alpha$ is an $r$-th power in $\FF_{q^s}$,  $N_{\FF_{q^s}/\FF_q}(\alpha)$ is an $r$-th power in $\FF_q$, and
    \item\label{part2} if $q\equiv 1 \mod r$ and $N_{\FF_{q^s}/\FF_q}(\alpha)$ is an $r$-th power in $\FF_q$, then $\alpha$ is an $r$-th power in $\FF_{q^s}$. 
    \end{enumerate}
\end{lemma}

\begin{proof}
    Consider the norm map $N_{\FF_{q^s}/\FF_q}: \FF_{q^s}\rightarrow \FF_q$. Part (\ref{part1}) follows immediately from the fact that this map is a homomorphism. Conversely, if $q\equiv 1 \mod r$, then exactly $1/r$ of the elements of $\FF_{q}^\times$ and $1/r$ of the elements of $\FF_{q^s}^\times$ are $r$-th powers. Since the map $N_{\FF_q^s/\FF_q}: \FF_{q^s}^\times\rightarrow \FF_q^\times$ is a surjective homomorphism of finite groups and $r$-th powers map to $r$-th powers, it follows that non-$r$-th powers map to non-$r$-th powers.
\end{proof}

We state below the theorem by G\'{o}mez-P\'{e}rez,  Nicol\'{a}s, Ostafe, and Sadornil that proves a necessary condition for dynamical irreducibility of polynomials. The proof uses Lemma \ref{Lemma:Capelli's} and norm computations, similar to the arguments in \cite{BJ}. They also use Stickelberger's Lemma~\cite{stickelberger} which describes a relationship between the parity of number of irreducible factors and the degree of the polynomial depending on the discriminant of the polynomial. The conditions they obtain are necessary for dynamical irreducibility, but not sufficient.

\begin{theorem}[Theorem 3.2, \cite{Ostafe_et_al.}] \label{thm:3.2ostafe}
Let $q =p^s$, $p$ an odd prime, and $f \in \FF_q[x]$ a dynamically irreducible polynomial with $\deg(f)\geq 2$, leading coefficient $a_d$, non-constant derivative $f'$ and $\deg(f')= k\leq d-1$. Then the following hold:
\begin{itemize}
    \item if $d$ is even, then $\Disc(f)$ and $a_d^k\Res(f^{(n)}, f')$, $n \geq 2$, are non-squares in $\FF_q$
    \item if $d$ is odd, then $\Disc(f)$ and $(-1)^{\frac{d-1}{2}}a_d^{(n-1)k+1}\Res(f^{(n)}, f')$, $n\geq 2$, are squares in $\FF_q$. 
\end{itemize}
\end{theorem}
Using formulas for the resultant, the expressions in Theorem \ref{thm:3.2ostafe} can be written explicitly in terms of the polynomial $f$ and points in the orbits of the critical points, as in \cite[Corollary 3.3]{Ostafe_et_al.}, we use such an expression in Section \ref{sec:Cubic polynomials}. 

\section{Unicritical Polynomials}\label{sec:unicrit}

A polynomial $f$ is \textit{unicritical} if it has a single affine critical point. 
Let $f$ be a unicritical polynomial of degree $d$. Let $a$ be the lead coefficient of $f$ and $\gamma$ the critical point of $f$. We first conjugate to move the critical point to $0$, let \[h(x)=f(x+\gamma)-\gamma.\] Then by Lemma \ref{Lemma: Conjugation}, $f$ is dynamically irreducible if and only if $(h,
-\gamma)$ is a dynamically irreducible pair. Since $h(x)$ is unicritical with critical point $0$, it has the form \[h(x)=ax^d+c. \]

The Vahlen-Capelli criterion provides irreducibility conditions for polynomials of the form $x^d+c.$  Naturally, this can be used to explore dynamical irreduciblity of unicritical polynomials. 
\begin{theorem}[Theorem 9.1, Chapter VI, \cite{lang02}]\label{Lemma: constant term rth power}
Let $K$ be a field and $d$ an integer $\geq 2$. Let $c\in K$, $c\neq 0$. Assume that for all prime numbers $r$ such that $r|d$ we have $c\notin K^r$, and if $4|d$ then $c\notin -4K^4$. Then $x^d-c$ is irreducible in $K[x]$.
\end{theorem}

In \cite{BJ} and \cite{BJErrata}, the authors give sufficient conditions for dynamical irreducibility of a quadratic polynomial, which are also necessary when working over a finite field. In this section we extend these results to all unicritical polynomials. 
\begin{proposition}[Proposition 2.3', \cite{BJErrata}] \label{prop: BJ2.3} Let $K$ have characteristic not equal to two.  A quadratic polynomial $f(x)=ax^2+bx+c \in K[x]$ is dynamically irreducible if $\lbrace -af(\frac{-b}{2a})\rbrace \cup  \lbrace af^n(\frac{-b}{2a}): n\geq 2\rbrace$ contains no squares. In the case $K$ is a finite field, $f$ is dynamically irreducible if and only if $\lbrace -af(\frac{-b}{2a})\rbrace \cup  \lbrace af^n(\frac{-b}{2a}): n\geq 2\rbrace$ contains no squares.
\end{proposition}

In the following lemma we generalize \cite[Lemma 2.5']{BJErrata}, in which the authors give conditions for the irreducibility of $g\circ f^n$ when $g\circ f^{n-1}$ is irreducible. These conditions are both necessary and sufficient in the finite field case, which allows us to give explicit dynamical irreducibility conditions for general unicritical polynomials.

\begin{lemma}\label{lem:BJ2.5ext} Let $K$ be a field with characteristic relatively prime to $d$.
    Let $h(x) = ax^d+c \in K[x]$.
    Let $g \in K[x]$  be a polynomial of degree $k\geq 1$ with leading coefficient $\ell(g)$. Suppose $g\circ h^{n-1}$ is irreducible over $K[x]$ for some $n \geq 1$. Let $C=(-a)^k\ell(g)$ if $n=1$ and $C=a^k\ell(g)$ if $n>1$ and let $D=(4a)^k\ell(g)$ if $n=1$ and $D=a^k\ell(g)$ if $n>1$. Then  $g\circ h^n$ is irreducible over $K$ if 
    \begin{itemize}
        \item $\frac{g\circ h^{n-1}(c)}{C}$ is not an $r$-th power in $K$ for every prime $r\mid d$, 
        \item and if $4\mid d$, $\frac{g\circ h^{n-1}(c)}{D}$ is not $4$-th power in $K$. 
    \end{itemize}  
     Moreover, in the case where $K$ is finite, $|K|\equiv 1 \mod r$ for each prime $r\mid d$, and if $4\mid d$, $|K|\equiv 1 \mod 4$, we have  $g\circ h^n$ is irreducible over $K$ if and only if the above conditions are satisfied.
 \end{lemma}

 \begin{proof}
 By hypothesis $g \circ h^{n-1}(x)$ is irreducible. Fix a root $\beta$ of $g\circ h^{n-1}$. Then by Capelli's Lemma, $g\circ h^n(x)$ is reducible over $K$ if and only if $h(x)-\beta$ is reducible over $K(\beta)$. By Theorem \ref{Lemma: constant term rth power}, $h(x)-\beta = ax^d+c-\beta$ is reducible over $K(\beta)$ if and only if $\frac{-(c-\beta)}{a}$ is an $r$-th power in $K(\beta)$ for some prime $r\mid d$ or if $4\mid d$ and  $\frac{(c-\beta)}{4a}$ is a $4$-th power in $K(\beta)$. 
 
 We first use the norm map to reduce to the conditions over $K(\beta)$ to conditions over $K$. If $\frac{-(c-\beta)}{a}$ is an $r$-th power in $K(\beta)$ for some prime $r\mid d$, then $N_{K(\beta)/K}\left(\frac{-(c-\beta)}{a}\right)$ is an $r$-th power in $K$ since the norm map is a multiplicative homomorphism. Similarly, if $4\mid d$ and  $\frac{(c-\beta)}{4a}$ is a $4$-th power in $K(\beta)$, then $N_{K(\beta)/K}\left(\frac{(c-\beta)}{4a}\right)$ is a $4$-th power in $K$. By Lemma \ref{lem:rthroot}, both of these implications become if and only if statements when the conditions $|K|\equiv 1 \mod r$ for each $r\mid d$, and if $4\mid d$, $|K|\equiv 1 \mod 4$ are satisfied.
 
 We now evaluate these norms using the fact that the Galois conjugates of $\beta$ are exactly the roots of $g\circ h^{n-1}$, 
\begin{align*}
    N_{K(\beta)/K}\left(\frac{-(c-\beta)}{a}\right) &= \left(\frac{-1}{a}\right)^{\deg(g\circ h^{n-1})}N_{K(\beta)/K}(c-\beta)\\
        &= \left(\frac{-1}{a}\right)^{kd^{n-1}}\frac{g\circ h^{n-1}(c)}{\ell(g\circ h^{n-1})}
, \end{align*}
where $\ell(g\circ h^{n-1})$ is the lead coefficient of $g\circ h^{n-1}$. Similarly,
\begin{align*}
    N_{K(\beta)/K}\left(\frac{(c-\beta)}{4a}\right) &= \left(\frac{1}{4a}\right)^{\deg(g\circ h^{n-1})}N_{K(\beta)/K}(c-\beta)\\
        &= \left(\frac{1}{4a}\right)^{kd^{n-1}}\frac{g\circ h^{n-1}(c)}{\ell(g\circ h^{n-1})}
, \end{align*}

When $n=1$, $\ell(g\circ h^{n-1})=\ell(g)$. When $n>1$, $\ell(g\circ h^{n-1}) = \ell(g)a^{k\left(\frac{d^{n-1}-1}{d-1}\right)}$, and we have $\frac{d^{n-1}-1}{d-1} \equiv 1 \mod r$ for each prime $r\mid d$ and  if $4\mid d$, $\frac{d^{n-1}-1}{d-1} \equiv 1 \mod 4$.
Thus, for each prime $r$ dividing $d$, 
$\left(\frac{-1}{a}\right)^{kd^{n-1}}\frac{g\circ h^{n-1}(c)}{\ell(g\circ h^{n-1})}$ is an $r$-th power in $K$ if and only if $\frac{g\circ h^{n-1}(c)}{(-1)^{kd^{n-1}}a^k\ell(g)}$ is an $r$-th power in $K$. Similarly, if $4\mid d$,  $\left(\frac{1}{4a}\right)^{kd^{n-1}}\frac{g\circ h^{n-1}(c)}{\ell(g\circ h^{n-1})}$ is a $4$-th power in $K$ if and only if $\frac{g\circ h^{n-1}(c)}{4^{kd^{n-1}}a^k\ell(g)}$ is a $4$-th power in $K$. 
\end{proof}

\begin{theorem}\label{prop:unicritical}
    Consider the polynomial $h(x)=ax^d+c\in \mathbb{F}_q[x]$. Suppose $q\equiv 1 \mod r$ for each $r\mid d$ and if $4\mid d$, suppose $q\equiv 1 \mod 4$. Then the pair $(h,\beta)$ is dynamically irreducible over $\mathbb{F}_q$ if and only if for every prime $r\mid d$
    \begin{itemize}
        \item $\frac{-(c-\beta)}{a}$ is not an $r$-th power in $\FF_q$, and 
        \item $\frac{h^{n-1}(c)-\beta}{a}$ is not an $r$-th power in $\FF_q$ for all $n>1$,
    \end{itemize}
        and if $4\mid d$,
    \begin{itemize}
        \item $\frac{(c-\beta)}{4a}$ is not an $4$-th power in $\FF_q$, and
        \item $\frac{h^{n-1}(c)-\beta}{a}$ is not an $4$-th power in $\FF_q$ for all $n>1$.
    \end{itemize} 
\end{theorem}

\begin{proof}
This follows by applying Lemma~\ref{lem:BJ2.5ext} with $g(x) = x-\beta$. 
\end{proof}

\begin{remark} If for any $r\mid d$ we have $q\not\equiv 1 \mod r$, then the map $x\mapsto x^r$ on $\FF_q$ is a bijection. Hence $x^r-\frac{(-(c-\beta))}{a}$ has a root in $\FF_q$ and by Theorem \ref{Lemma: constant term rth power}, $h(x)-\beta$ is reducible.  Similarly, if $4\mid d$ and $q\not\equiv 1\mod 4$, then the map $x\mapsto x^4$ is a bijection and $h(x)-\beta$ is reducible. Therefore, $(h,\beta)$ is not dynamically irreducible if the hypotheses $q\equiv 1 \mod r$ for each $r\mid d$ and if $4\mid d$, $q\equiv 1 \mod 4$ are not met.
\end{remark}

We now prove Theorem \ref{thm:maintheorem} as a special case of Theorem \ref{prop:unicritical}.

\begin{proof}[Proof of Theorem~\ref{thm:maintheorem}]
    Let $f(x)\in \FF_q$ be a unicritical point of prime degree $r$ with critical point $\gamma$ and lead coefficient $a$. Suppose $q\equiv 1 \mod r$. Let $h(x)=f(x+\gamma)-\gamma$. Then by Lemma \ref{Lemma: Conjugation}, $f$ is dynamically irreducible if and only if $(h,-\gamma)$ is dynamically irreducible. By Theorem \ref{prop:unicritical}, this holds if and only if $\frac{-(c+\gamma)}{a}$ is not an $r$-th power in $\FF_q$, and $\frac{h^{n-1}(c)+\gamma}{a}$ is not an $r$-th power in $\FF_q$ for all $n>1$. Note, $\frac{-(c+\gamma)}{a}=\frac{-(h(0)+\gamma)}{a}=\frac{-f(\gamma)}{a}$ and  $\frac{h^{n-1}(c)+\gamma}{a}=\frac{h^{n}(0)+\gamma}{a}=\frac{f^n(\gamma)}{a}$.
\end{proof}

Similar to \cite{BJ}, we can define the \emph{adjusted critical orbit} of a critical point $\gamma$ of a  unicritical polynomial $f$ as the set from Theorem \ref{thm:maintheorem},
\[
\left\lbrace -\frac{f(\gamma)}{a}\right\rbrace \bigcup \left\lbrace \frac{f^n(\gamma)}{a}: n\in \mathbb{Z}_{>1}\right\rbrace.
\] 
We observe that our definition differs slightly from \cite{BJ}, as their definition has the lead coefficient $a$ as a multiplicative factor, however the definitions are equivalent in the quadratic case. With this definition, Theorem \ref{thm:maintheorem} says that a unicritical polynomial $f(x)\in \FF_q$ of prime degree $r$ with $q\equiv 1\mod r$ is dynamically irreducible if and only if its adjusted critical orbit contains no $r$-th powers.

The techniques and results above are directly inspired by \cite{BJ}. In fact, we can recover \cite[Lemma 2.5, Proposition 2.3] {BJ} (Proposition \ref{prop: BJ2.3} above).
\begin{remark}
    Let $f(x) = ax^2+bx+c \in \mathbb{F}_q[x]$ where $q$ is an odd prime. Then the critical point of $f$ is $\frac{-b}{2a}$. Hence $f$ is dynamically irreducible if and only if $-a^{-1}f(\frac{-b}{2a})$ is not a square and $a^{-1}f^{n}\left(\frac{-b}{2a}\right)$ is not a square for $n>1$. By multiplying by $a^2$, we recover the conditions in Proposition \ref{prop:  BJ2.3}.
\end{remark}
In the following example, we apply Theorem \ref{thm:maintheorem} to find dynamically irreducible unicritical cubic polynomials over $\FF_7$. In this example, and in examples in forthcoming sections, we will display the critical portrait, which a directed graph on the sequence of iterates $\left( f^n(\gamma)\right)_{n\in\ZZ_\geq0}$ for each critical point $\gamma$ of $f$. 
\begin{example} \label{ex:unicrit}
Consider $f(x) = x^3+3\in \FF_7$. Given the critical portrait in Figure \ref{fig:exampleunicrit} we have that the adjusted critical orbit of $f$ is $\lbrace 2,4\rbrace$, which contains no cubes and thus by Theorem~\ref{thm:maintheorem} we have that $f(x) = x^3+3$ is dynamically irreducible.
\begin{center}
\begin{figure}[h]
    \centering
\begin{tikzpicture}
    \node (c) at (0,0) {$0$};
    \node (fc) at (1,0) {$3$};
    \node (f2c) at (2,0) {$2$};
    \node (f3c) at (3,0) {$4$};
    \draw[->] (c) --node[above] {$f$} (fc);
    \draw[->] (fc) --node[above] {$f$} (f2c);
    \draw[->] (f2c) --node[above] {$f$} (f3c);
    \draw [->] (f3c) to [out=-30,  in=30, looseness = 6]node[right] {$f$} (f3c);
\end{tikzpicture}
    \caption{}
    \label{fig:exampleunicrit}
\end{figure}
\end{center}

In fact, there are exactly 12 polynomials of the form $f(x) = ax^3+c \in \FF_7[x]$ with $a,c \neq 0$ that are dynamically irreducible. They are listed in Table \ref{tab:F7example} below along with their adjusted critical orbits.
    \def\arraystretch{1.9}
    \begin{table}[h!]
    \centering
    \caption{}
    \label{tab:F7example}

    \begin{tabular}{c|c}
     Adjusted critical orbit & Polynomials $ax^3+c \in \FF_7$  \\\hline
     $\lbrace 2,4 \rbrace$ & \makecell{$x^3+3, 2x^3+6, 4x^3+5$\\
     $3x^3+2, 5x^3+1, 6x^3+4$}\\\hline
     $\lbrace 3,5 \rbrace$ & \makecell{$x^3+4, 2x^3+1, 4x^3+2$\\$3x^3+5, 5x^3+6, 6x^3+3$}\\
\end{tabular}
\end{table}
\end{example}

\section{Cubic Polynomials}\label{sec:Cubic polynomials}

In this section, we examine the dynamical irreducibility of cubic polynomials over finite fields. Let $q=p^s$ for some odd prime $p$ and positive integer $s$ and consider 
\begin{align}
    f(x)=a_3x^3+a_2x^2+a_1x+a_0 \in \FF_q,\label{eqn:gen_peng}
\end{align}
where $a_3\neq 0$.

We first apply the result of G\'{o}mez-P\'{e}rez,  Nicol\'{a}s, Ostafe, and Sadornil given in Theorem \ref{thm:3.2ostafe} to $f$, and in particular, the formulas given in \cite[Corollary 3.3]{Ostafe_et_al.}. Let $k$ be the degree of $f'$ and $\gamma_1, \gamma_2$ the critical points of $f$. If $f(x)$ is dynamically irreducible over $\FF_q$, then  \[(-1)^{k+1}(k+1)a_{k+1}a_3^{(n-1)k+1}f^n(\gamma_1)f^n(\gamma_2)\] is a square in $\FF_q$ for all $n$. 
If $\FF_q$ has characteristic $p\geq 5$, for each $n$, this quantity is a square if and only if the quantity 
\begin{equation}\label{cubicdiscnorm} -3f^n(\gamma_1)f^n(\gamma_2)\end{equation} 
is a square, since in this case $k=2$.

Since this condition is necessary, but not sufficient, we investigate further using the following result of Dickson. 

\begin{theorem}[Theorem 3, \cite{Dickson}] \label{thm:Dickson} Let $p>3$ be an odd prime and $q=p^s$. Suppose $a_0a_1\neq 0$. The polynomial \[x^3+a_1x+a_0\] is irreducible over $\FF_q$ if and only if the following conditions hold
\begin{enumerate}
    \item $-4a_1^3-27a_0^2$ is a non-zero square in $\FF_q$, and
    \item for $\mu\in \FF_q$ satisfying $-4a_1^3-27a_0^2=81\mu^2$, $\frac{1}{2}(-a_0+\mu\sqrt{-3})$ is not a cube in $\FF_q(\sqrt{-3})$.
\end{enumerate}
\end{theorem} 

\begin{example}\label{ex:(6,2)1st_iterate}
    Consider the polynomial $h(x)=x^3+6x+2$ over $\FF_7$. Letting $a_0=2$ and $a_1=6$ in the first condition, we obtain
    \[-4\cdot6^3-27\cdot2^2\equiv 1 \pmod{7},\]
    which is a non-zero square modulo $7$. We can then determine $\mu$ such that $1 \equiv 81\mu^2 \pmod{7}$ by noting that $81^{-1} \equiv 2 \pmod{7}$, so such a $\mu$ satisfies $\mu^2 \equiv 2 \pmod{7}$, and thus $\mu \in \{3, 4\}$.

For each $\mu \in \lbrace 3,4 \rbrace$ we can compute that $\frac{1}{2}(-2+\mu\sqrt{-3})\equiv 2$ or $3\pmod{7}$. Since neither $2$ nor $3$ are cubes in $\FF_7(\sqrt{-3})$, we can conclude that this cubic is irreducible over $\FF_7$ by Theorem \ref{thm:Dickson}.        
\end{example}

\begin{remark} When $a_1=0$, the polynomial is unicritical and we may apply Theorem \ref{Lemma: constant term rth power}. Theorem \ref{thm:Dickson} also applies when $a_1=0$, after a minor adjustment. In this case the quantity in the second condition simplifies to $\frac{1}{2}\left(-a_0+\sqrt{a_0^2}\right)$. We must choose $\sqrt{a_0^2}=-a_0$ so that this quantity is non-zero.  Then the second condition reduces to the requirement that $-a_0$ is not a cube, which matches the criterion for uni-critical polynomials from Theorem \ref{Lemma: constant term rth power}.
\end{remark}

With $f$ as in Equation \eqref{eqn:gen_peng}, we make a standard change of variables; 
\begin{equation}\label{eq:cubiconjugate}
h(x)=f\left(x-\frac{a_2}{3a_3}\right)+\frac{a_2}{3a_3} = b_3x^3+b_1x+b_0,\end{equation}
where $b_3=a_3$, $b_1=a_1-\frac{a_2^2}{3a_3}$, and $b_0 = a_0 + \frac{2a_2^3-9a_3a_2(a_1-1)}{27a_3^2}$. Note, if the lead coefficient of $f$ is a square, we can make a further change of variables to get a monic polynomial conjugate to $f$, as in Example \ref{ex:cubicleadcoeff}. Since this is not always possible, we avoid doing so here.

By Lemma \ref{Lemma: Conjugation}, $f^n(x)$ is irreducible if and only if $h^n(x)-\frac{a_2}{3a_3}$ is irreducible. Proposition \ref{prop:cubic} gives criterion for dynamical irreducibility of a pair $(h,\beta_0)$, with $h$ of the form $h(x) = b_3x^3+b_1x+b_0$, using Theorem \ref{thm:Dickson}. This can be used to test dynamical irreducibility of $f$ by setting $\beta_0 = \frac{a_2}{3a_3}$. 
We now prove Proposition \ref{prop:cubic}. 

\begin{proof}[Proof of Proposition~\ref{prop:cubic}]
 Suppose $h^{n}(x)-\beta_0$ is irreducible over $\FF_q$ for some $n\geq 0$. Then by Cappelli's Lemma (Lemma \ref{Lemma:Capelli's}), $h^{n+1}(x)-\beta_0$ is irreducible over $\FF_q$ if and only if $h(x)-\beta_{n}$ is irreducible over $\FF_q(\beta_{n})$. Applying Theorem \ref{thm:Dickson} to $\frac{1}{b_3}(h(x)-\beta_{n})$, we see $h(x)-\beta_{n}$ is irreducible over $\FF_q(\beta_{n})$ if and only if 
    \[-4\left(\frac{b_1}{b_3}\right)^3-27\left(\frac{b_0-\beta_n}{b_3}\right)^2\] is a non-zero square in $\FF_q(\beta_n)$, and for $\mu_n\in \FF_q(\beta_n)^\times$ such that $-4\left(\frac{b_1}{b_3}\right)^3-27\left(\frac{b_0-\beta_n}{b_3}\right)^2=81\mu^2$, we have \[\frac{1}{2}\left(-\left(\frac{b_0-\beta_n}{b_3}\right)+\mu_n\sqrt{-3}\right)\] is not a cube in $\FF_q(\sqrt{-3},\beta_n)$.
 Using the formulas for the discriminant and resultant given in Section \ref{sec:background}, we can see, $-4\left(\frac{b_1}{b_3}\right)^3-27\left(\frac{b_0-\beta_n}{b_3}\right)^2 = -\Res\left(\frac{1}{b_3}(h(x)-\beta_n), \frac{1}{b_3}h'(x)\right)= \Disc\left(\frac{1}{b_3}(h(x)-\beta_n)\right)$. 

 By Lemma \ref{lem:rthroot}, $\Disc\left(\frac{1}{b_3}(h(x)-\beta_n)\right)$ is a square in $\FF_q(\beta_n)$ if and only if $$N_{\FF_{q^{3^n}}/\FF_q}\left(\Disc\left(\frac{1}{b_3}(h(x)-\beta_n)\right)\right)$$ is a square in $\FF_q$. Letting $\gamma_1, \gamma_2$ denote the critical points of $h$ and $B$ the lead coefficient of $h^n$. Then 
\begin{eqnarray*}
    N_{\FF_{q^{3^n}}/\FF_q}\left(\Disc\left(\frac{1}{b_3}(h(x)-\beta_n)\right)\right)&= &N_{\FF_{q^{3^n}}/\FF_q}\left((-1)\Res\left(\frac{1}{b_3}(h(x)-\beta_n),\frac{1}{b_3}h'(x)\right)\right) \nonumber \\
    &= &N_{\FF_{q^{3^n}}/\FF_q}\left(-3^3\cdot \frac{h(\gamma_1)-\beta_n}{b_3}\cdot\frac{h(\gamma_2)-\beta_n}{b_3}\right) \nonumber\\
    &= & N_{\FF_{q^{3^n}}/\FF_q}\left(\frac{-27}{b_3^2}\right)N_{\FF_{q^{3^n}}/\FF_q}\left(h(\gamma_1)-\beta_n\right)N_{\FF_{q^{3^n}}/\FF_q}\left(h(\gamma_2)-\beta_n\right)\nonumber \\
    &= &\left(\frac{-27}{b_3^2}\right)^{3^{n-1}}\frac{1}{B^2}(h^{n+1}(\gamma_1)-\beta_0)(h^{n+1}(\gamma_2)-\beta_0). 
\end{eqnarray*}
This quantity is a square if and only if $-3(h^{n+1}(\gamma_1)-\beta_0)(h^{n+1}(\gamma_2)-\beta_0)$ is a square. 
\end{proof}

Returning to the general setting, we have $h(x)=f(x-\beta_0)+\beta_0$ with $\beta_0=\frac{a_2}{3a_3}$. Then \[-3(h^n(c_1)-\beta_0)(h^n(c_2)-\beta_0)=-3f^n(\gamma_1-\beta_0)f^n(\gamma_2-\beta_0).\] Since the critical points of $f(x)$ are $\gamma_1-\beta_0$ and $\gamma_2-\beta_0$, Condition \ref{cond1} reduces to the condition given in Equation \eqref{cubicdiscnorm}, as expected. 

\begin{remark}
        With set up as is Proposition \ref{prop:cubic}, assume $-3(h^n(\gamma_1)-\beta_0)(h^n(\gamma_2)-\beta_0)$ is a square for all $n\geq 1$, so Condition \ref{cond1} of Proposition \ref{prop:cubic} holds. Then the pair $(h,\beta_0)$ is dynamically irreducible if and only if Condition \ref{cond2} of Proposition \ref{prop:cubic} holds. By Lemma \ref{lem:rthroot}, $\frac{1}{2}\left(-\left(\frac{b_0-\beta_n}{b_3}\right)+\mu_n\sqrt{-3}\right)$ is not a cube in $\FF_q(\sqrt{-3},\beta_n)$ if and only if its norm is not a cube in $\FF_q$, however, it is not clear to the authors how to write a general formula for the norm of this expression. The condition can be checked up to $n=10$ fairly quickly using Sage for $p \equiv 1 \mod 3$. 
        (See Algorithm~\ref{alg:peng_algorithm} in Appendix~\ref{appendix}.) 
\end{remark}
In the following examples we demonstrate how Proposition \ref{prop:cubic} applies to specific cubic polynomials.
We note that via the proof of Propoisiton \ref{prop:cubic}, if $h^n(x)$ is irreducible, then $h^{n+1}$ is irreducible if and only if conditions \ref{cond1} and \ref{cond2} are satisfied. 

\begin{example}\label{ex:cubiccubes}
In this example, we apply Proposition~\ref{prop:cubic} to test the irreducibility of iterates of the cubic polynomial  $h(x)=x^3+6x+2$ over $\FF_7$. First, recall that in Example~\ref{ex:(6,2)1st_iterate}, we showed that $h(x)$ is irreducible using Theorem~\ref{thm:Dickson}. We will show that $h^{n+1}(x)$ satisfies the conditions of Proposition \ref{prop:cubic} for $n\in \lbrace 1,2,3,4 \rbrace$ however, $h(x)$ will be reducible at the fifth iterate, and in particular it fails the cube condition. 

We compute the set of critical points, $\{\gamma_1,\gamma_2\}$, of $h(x)$ as $\{\sqrt{5},-\sqrt{5}\}$ the critical orbits of $h(x)$ is given below in Figure \ref{fig:examplecubic1portrait}. 
\begin{center}
\begin{figure}[h]
\begin{tikzpicture}
    \node(c) at (-.3,0) {$\pm \sqrt{5}$};
    \node (fc) at (2,0) {$\pm4\sqrt{5}+2$};
    \node (f2c) at (4,0) {$5$};
    \node (f3c) at (5.5,0) {$3$};
    \draw[->] (c) --node[above] {$h$} (fc);
    \draw[->] (fc) --node[above] {$h$} (f2c);
    \draw[->] (f2c) to [out=-60,  in=-120]node[above] {$h$} (f3c);
    \draw [->] (f3c) to [out=120,  in=60]node[above] {$h$} (f2c);
\end{tikzpicture}
    \caption{}
    \label{fig:examplecubic1portrait}
\end{figure}
\end{center}
Since $\left( -3h^{n}(\gamma_1)h^{n}(\gamma_2)\right)_{n \in \ZZ_{\geq 1}}  = \left( 4,2,1,2,1\ldots \right) \subset \FF_7^2$, then Condition \ref{cond1} of Proposition \ref{prop:cubic} is always satisfied. 

 Now, let $\beta_1$ be a root of $h(x)$. To determine $\mu_1\in \FF_7[\beta_1]$ we first note that requiring $\mu$ such that 
 \[81\mu^2 \equiv -4(6)^3-27(2-\beta_1)^2 \pmod 7\]
 is equivalent to 
 \[\mu^2 \equiv 2-\beta_1+2\beta_1^2 \pmod{7}.\] We see that $\mu = 2\beta_1^2+3 \mod 7$ satisfies this relationship. Then, by setting $\sqrt{-3} = 2$ in $\FF_7$ (or, $\sqrt{-3} = 5$) we compute that 
    \[\frac{1}{2}(-(2-\beta_1)+(2\beta_1+3)\sqrt{-3})\equiv \frac{1}{2}(-(2-\beta_1)+(2\beta_1^2+3)\sqrt{-3})\pmod{7}\]
    which simplifies to $2\beta_1^2-3\beta_1+2$ (or $2\beta_1^2-\beta_1+2$, respectively).
     Plugging this value into the norm function yields
    \[N(2\beta_1^2-3\beta_1+2)=(2\beta_1^2-3\beta_1+2)(2\beta_1^2-3\beta_1+2)^7(2\beta_1^2-3\beta_1+2)^{7^2}\equiv 4\pmod{7},\]
    (or $5$, respectively) which is not a cube in $\FF_7$. Therefore, $h^2(x)$ is irreducible by Proposition \ref{prop:cubic}. We can similarly compute $N(\frac{1}{2}(-(2-\beta_n)+(\mu_n)2)\pmod{7})$ and get a list of norms $(5,2,2,1)$ for up to $n=4$. Notice that $1$ is a cube implying that $\frac{1}{2}(-(2-\beta_4)+\mu_4\sqrt{-3})$ is a cube in $\FF_7(\sqrt{-3},\beta_4)$, and thus $h^5(x)$ is reducible. However, we have that $h^n(x)$ is  irreducible for $n=1,2,3,4$.
\end{example}

\begin{example}\label{ex: cubicsquares} In this example we will use Proposition \ref{prop:cubic} to show that  $h(x)=x^3+x+1$ is not dynamically irreducible over $\FF_{19}$. By Theorem \ref{thm:Dickson}, we have that $h(x)$ is irreducible. First, $-4(1)^3-27(1)^2 \equiv 7 \mod{19}$, which is a (non-zero) square in $\FF_{19}$. Setting up the second condition, we see that $81\mu^2 \equiv 7 \pmod{19}$ is satisfied by $\mu \in \lbrace 3, 16 \rbrace $. For either value of $\mu$, and either embedding of $\sqrt{-3}$ into $\FF_{19}$ we have that $\frac{1}{2}(-1+\mu\sqrt{-3}) \in \lbrace 15, 3 \rbrace$, neither of which are cubes in $\FF_{19}$, and thus $h(x)$ is irreducible. 
    
To determine the irreducibility of $h^2(x)$ we will use Proposition \ref{prop:cubic} with $\beta_0 = 0$. Let $\beta_1$ be a root of $h(x)$. The critical points of $h(x)$ are $\{\gamma_1,\gamma_2\}=\{5,14\}$ in $\FF_{19}$ and the critical orbit is shown in Figure \ref{fig:examplecubic2portrait} below.
\begin{center}
\begin{figure}[h]
\begin{tikzpicture}
    \node(c) at (0,0) {$5$};
    \node (fc) at (1.5,0) {$17$};
    \node (f2c) at (3,0) {$10$};
    \node (f3c) at (4.5,0) {$4$};
    \node (f4c) at (6,0) {$12$};
    \node (c2) at (4.5,1.5) {$-5$};
    \draw[->] (c) --node[above] {$h$} (fc);
    \draw[->] (fc) --node[above] {$h$} (f2c);
    \draw[->] (f2c) --node[above] {$h$} (f3c);
    \draw[->] (c2) --node[right] {$h$} (f3c);
    \draw[->] (f3c) --node[above] {$h$} (f4c);
    \draw[->] (f4c) to [looseness=4, out=-30,  in=30]node[right] {$h$} (f4c);
\end{tikzpicture}
    \caption{}
    \label{fig:examplecubic2portrait}
\end{figure}
\end{center}
    Then for condition \ref{cond1} for we have that
    \[\left( -3h^{n+1}(\gamma_1) h^{n+1}(\gamma_2)\right)_{n\in \ZZ_{\geq 0}}= \left( 5,1,8,5,5, \ldots\right) \]
    We note that $8$ is not a square in $\FF_{19}$, and therefore the first condition of Proposition \ref{prop:cubic} fails at $n=2$ so we know that $h^3(x)$ is reducible over $\FF_{19}$.  
\end{example}

\begin{remark}\label{rmk:reducibility}
        We apply similar techniques to those used in Examples \ref{ex:cubiccubes} and \ref{ex: cubicsquares} to implement Algorithm~\ref{alg:peng_algorithm} to determine if arbitrary cubic polynomials are potentially dynamically irreducible. Using Algorithm~\ref{alg:peng_algorithm}, along with other computational techniques, we obtain that there are no dynamically irreducible cubics of the form $h(x) = x^3+b_1x+b_0$ over $\FF_p$ for 
        $p\in \lbrace 5, 11, 13, 17, 23, 31, 41 \rbrace$. We note that in many cases looking for squares in the sequence $\left(-3h(\gamma_1)h(\gamma_2)\right)_{n\geq 1}$ shows that $h(x)$ cannot be dynamically irreducible. 
\end{remark}

There is one family of polynomials for which it is possible to apply Theorem \ref{thm:Dickson} to determine dynamical irreducibility working only over $\FF_q$. In particular, Chu proved the following, which gives a family of dynamically irreducible polynomials after conjugation.

\begin{theorem}[Theorem 3, \cite{Chu}]\label{thm:Chu} Let $q=p^s$, $p>3$ and $g(x)\neq x\pm 2$ be an irreducible polynomial over $\FF_q$ with degree $m$. Define $f_0(x)=g(x)$ and $f_{k+1}(x)=f_k(x^3-3x)$ for all $k\geq 0$. Then $f_0(x), f_1(x), f_2(x),\dots$ are all irreducible over $\FF_q$ if and only if for any root $\alpha$ of $g(x)$ we have $-3(\alpha^2-4)$ is a nonzero square in $\FF_{q^m}$ and the roots of $x^2-\alpha x+1$ are not cubes in 
$\FF_q(\sqrt{-3})$.
\end{theorem}

\begin{proposition}\label{prop:chu_irred} Let $q=p^s$, $p>3$. Then $x^3+3\alpha x^2+(3\alpha^2-3)x+(\alpha^3-4\alpha)$ with $\alpha\neq \pm 2$ is dynamically irreducible over $\FF_q$ if and only if $-3(\alpha^2-4)$ is a nonzero square in $\FF_{q}$ and the roots $\frac{1}{2}(\alpha\pm\sqrt{\alpha^2-4})$ of $x^2-\alpha x+1$ are not cubes in 
$\FF_q(\sqrt{-3})$.
\end{proposition}

Note, when $\alpha =\pm 2$, $x^3+3\alpha x^2+(3\alpha^2-3)x+(\alpha^3-4\alpha)=x^3\pm 6 x^2+9x$ which is reducible since $0$ is a root.

\begin{proof}
Let $f(x)=x^3-3x$ and $g(x)=x-\alpha$ with $\alpha\neq \pm 2$. By Theorem \ref{thm:Chu}, the pair $(f,\alpha)$ is dynamically irreducible over $\FF_q$ if and only if  $-3(\alpha^2-4)$ is a nonzero square in $\FF_{q}$ and the roots of $x^2-\alpha x+1$ are not cubes in 
    \[\FF_q(\sqrt{-3})=\begin{cases}
        \FF_{q} \text{ if } 3\mid (q-1),\\
        \FF_{q^{2}} \text{ if } 3\mid (q-2).
    \end{cases}\]

Using Lemma \ref{Lemma: Conjugation}, $f(x+\alpha)-\alpha = x^3+3\alpha x^2+(3\alpha^2-3)x+(\alpha^3-4\alpha)$ is dynamically irreducible if and only if the pair $(f,\alpha)$ is dynamically irreducible.
\end{proof}

\begin{example}\label{ex:Propcubic}
Working over $\FF_7$ we can use Proposition~\ref{prop:chu_irred} to show that $x^3\mp 3x^2\pm 3$ is dynamically irreducible over $\FF_7.$ In fact they are the only cubic polynomials in the form of Proposition \ref{prop:chu_irred} over $\FF_7$ that are dynamically irreducible. 
\end{example}

\begin{example}\label{ex:Chu}
Table~\ref{tab:cubeover3} shows data for reducibility of iterates of cubics over $\FF_3.$ Here we don't use our algorithm since it only applies to $p>3,$ and instead rely on built in SageMath functions that are using the modular arithmetic algorithms. Through this we can see that there are no dynamically irreducible cubics over $\FF_3$. Table~\ref{tab:cubeover3} also includes information about the smallest reducible iterate of each polynomial. 
 \def\arraystretch{1.2}
    \begin{table}[h]
\centering
\caption{
}
    \begin{tabular}{c|c}

    Polynomials & $\min\lbrace n:f^n(x) \text{ is reducible}\rbrace$   \\\hline

        $x^3 + 2x + 1$,
 $x^3 + 2x + 2$,
 $x^3 + x^2 + 2$,   & \multirow{4}{*}{$2$}
 \\

 $x^3 + x^2 + x + 2$,
 $x^3 + 2x^2 + 1$,
 $x^3 + 2x^2 + x + 1$, & \\
 $2x^3 + x + 1$,
 $2x^3 + x + 2$,
 $2x^3 + x^2 + x + 1$, & \\
 $2x^3 + x^2 + 2x + 2$,
 $2x^3 + 2x^2 + x + 2$,
 $2x^3 + 2x^2 + 2x + 1$ &  \\ \hline
      $x^3+x^2+2x+1$, $x^3+2x^2+2x+2$   & $3$ \\ \hline
     $2x^3+x^2+2$, $2x^3+2x^2+1$    & $4$ \\
    \end{tabular}\label{tab:cubeover3}
    \end{table}
\end{example}

 \section{Shifted Linearized Polynomials over Finite Fields}\label{Sec:Artin}

In this section, we explore dynamical irreducibility of shifted linearized polynomials  over finite fields. A \textit{linearized polynomial} over $\FF_q$ with $q=p^s$ is a polynomial of the form 
\[a_nx^{p^n}+a_{n-1}x^{p^{n-1}}+\dots+a_1x^p+a_0x.\]
These polynomials are called linearized because they represent a linear mapping on $\FF_q$.
A theorem of Agou discussed in \cite{cohen_1989} implies a shifted linearized polynomial of degree $p^n$ is reducible unless $n=1$ or $p=n=2$. Thus, iterates of shifted linearized polynomials are always dynamically reducible; in fact the second iterate is reducible unless $p=\deg f=2$, and in this case the third iterate must be reducible. In \cite{Ahmadi}, Ahmadi gives an alternate proof that the third iterate of a quadratic polynomial must be reducible over a field of characteristic $2$. In \cite{Ostafe_et_al.}, the authors apply this argument to show the third iterate of a shifted linearized cubic must be reducible over a field of characteristic $3$. We push this argument further to show a shifted linear degree $p$ polynomial over a field of characteristic $p$ must have reducible second iterate unless $p=2$, in which case the polynomial must have reducible third iterate. As a consequence, it follows that no Artin-Schreier polynomials are dynamically irreducible.  

Consider $f(x)=a_px^p-a_1x-a_0$ over a finite field of characteristic $p$. The following lemma provides necessary and sufficient conditions for when a polynomial of this form is irreducible. 

\begin{lemma}[\cite{cohen_1989}, Lemma 2]\label{lemma:Cohen}
    Let $q=p^s$. The polynomial $f(x)=x^p-a_1x-a_0$ is irreducible over $\FF_q$ if and only if $a_1=A^{p-1}$ for some $A\in \FF_q$ and $\Tr_{\FF_{q}/\FF_{p}}\left(\frac{a_0}{A^p}\right)\neq 0$.
\end{lemma}

We use this lemma to prove the main result in this section.

\begin{theorem}\label{Thm:Artin-Schreier}
    Let $p$ be a prime and $q=p^s$ for $s$ a positive integer. If $f(x)=a_px^p-a_1x-a_0\in \FF_q[x]$ with $a_p a_1\neq 0$, then $f^2(x)$ is reducible if $p\geq 3$ and $f^3(x)$ is reducible if $p=2$. 
\end{theorem}
\begin{proof} Note if $f^2(x)$ is reducible we are done. Suppose $f^2(x)$ is irreducible, let $\gamma$ be any root of $f(x)$. By Capelli's Lemma (Lemma \ref{Lemma:Capelli's}), $f(x)$ is irreducible over $\FF_q$ and $f(x)-\gamma$ is irreducible over $\FF_q(\gamma) = \FF_{q^p}$.  
Since $f(x)$ is irreducible, $h(x)=x^p-\frac{a_1}{a_p}x-\frac{a_0}{a_p}$ is also irreducible over $\FF_q$ so $\frac{a_1}{a_p}=A^{p-1}$ for $A\in \FF_{q}$ and $\Tr_{\FF_{q}/\FF_{p}}\left(\frac{a_0}{a_pA^p}\right)\neq 0$ by Lemma \ref{lemma:Cohen}. Also, since $f(x)-\gamma$ is irreducible over $\FF_{q^p}$, so is $x^p-\frac{a_1}{a_p}x-\frac{a_0+\gamma}{a_p}$ and hence by Lemma \ref{lemma:Cohen},  $\Tr_{\FF_{q^p}/\FF_{p}}\left(\frac{a_0+\gamma}{a_pA^p}\right)\neq 0$. 
 
Using properties of trace from Theorem \ref{Lemma: Transitivity of Norm and Trace}, we have 
    \begin{align*}
        \Tr_{\FF_{q^p}/\FF_{p}}\left(\frac{a_0+\gamma}{a_pA^p}\right)&=\Tr_{\FF_{q^p}/\FF_{p}}\left(\frac{a_0}{a_pA^p}\right)+\Tr_{\FF_{q^p}/\FF_{p}}\left(\frac{\gamma}{a_pA^p}\right)\\
        &=\Tr_{\FF_{q}/\FF_{p}}\left(\Tr_{\FF_{q^p}/\FF_{q}}\left(\frac{a_0}{a_pA^p}\right)\right)+\Tr_{\FF_{q^p}/\FF_{p}}\left(\frac{\gamma}{a_pA^p}\right)\\
        &=\Tr_{\FF_{q}/\FF_{p}}\left(0\right)+\Tr_{\FF_{q^p}/\FF_{p}}\left(\frac{\gamma}{a_pA^p}\right)\\
        &=\Tr_{\FF_{q^p}/\FF_{p}}\left(\frac{\gamma}{a_pA^p}\right).
    \end{align*}
    Hence, $\Tr_{\FF_{q^p}/\FF_{p}}\left(\frac{\gamma}{a_pA^p}\right)\neq 0$.

    Since $f(x)$ is irreducible over $\FF_q$, it is the minimal polynomial for $\gamma$ over $\FF_q$. When $p\geq 3$, this implies $\Tr_{\FF_{q^p}/\FF_{q}}\left(\frac{\gamma}{a_pA^p}\right)=0$ since the coefficient of the $p-1$ term of $f(x)$ is $0$ and thus,
    \[\Tr_{\FF_{q^p}/\FF_{p}}\left(\frac{\gamma}{a_pA^p}\right) = \Tr_{\FF_{q}/\FF_{p}}\left(\Tr_{\FF_{q^p}/\FF_{q}}\left(\frac{\gamma}{a_pA^p}\right)\right)=\Tr_{\FF_{q}/\FF_{p}}(0)=0,\]
    which is a contradiction.

    When $p=2$, let $\gamma'$ be a root of $f^2(x)$, then by a similar argument, if $f^3(x)$ is irreducible over $\FF_{q}$, then $f^2(x)-\gamma'$ is irreducible over $\FF_{q^4}$, so $\Tr_{\FF_{q^4}/\FF_q}\left(\frac{a_0+\gamma'}{a_2A^2}\right)\neq 0$, which implies $\Tr_{\FF_{q^4}/\FF_q}\left(\frac{\gamma'}{a_2A^2}\right)\neq 0$. On the other hand, since $f^2(x)$ is the minimal polynomial for $\gamma'$, $\Tr_{\FF_{q^4}/\FF_q}\left(\frac{\gamma'}{a_2A^2}\right)= 0$, which is a contradiction.
    \end{proof}

\textbf{Acknowledgements.} This work is based on a project that started during the Rethinking Number Theory 4 workshop. The authors thank the organizers for bringing the group together and facilitating the project. The authors would also like to thank Rafe Jones and Vefa Goksel for helpful feedback. The sixth author's research is partially supported by an AMS-Simons Travel Grant. 

\appendix
\section{Algorithm for Checking Irreducibility of an Iterated Cubic Polynomial over a Finite Field}\label{appendix}

\begin{algorithm}[H]
 \KwData{Takes in a cubic of the form $h(x)=x^3+b_1x+b_0$ over $\FF_p[x]$, $p\equiv 1\pmod{3}$.}
 \KwResult{Outputs dynamical irreducibility information for cubic $h(x)=x^3+b_1x+b_0$ over a given number of iterates}
 initialization\;
 initialize $g(x)=x$\;
 initialize counter=0\;

 \For{ i in range of iterates}{
  Find a root $\beta$ of $g$ over an extension field of $\FF_p[x]$\;  
  Create square\_check\_object$=-4b_1^3-27(b_0-\beta)^2$  \;
  \eIf{norm(square\_check\_object) is a square}{
    Create $\mu$=squareroot(square\_check\_object/81)\; 
   Create cube\_check\_object=$\frac{1}{2}(-(b_0-\beta)+\mu\sqrt{-3})$\;
   
   \eIf{norm(cube\_check\_object) is not a cube}{
   $h^{\text{counter}}$ is irreducible\;
   set counter=+1\;
   set g=h(g)\;
   go back to the beginning of the for loop section\;
    }
   {
  
   $h^{\text{counter}}$ is reducible \;
   }
   }
   {$h^{\text{counter}}$ is reducible\;
 }
 }
 \caption{A description of how to implement Proposition~\ref{prop:cubic} in a program. The algorithm was tested in SageMath~\cite{sagemath}.}\label{alg:peng_algorithm}
\end{algorithm}

\begin{remark} 
    Algorithm~\ref{alg:peng_algorithm} can be implemented for odd primes $p\equiv 2 \mod 3$, however it requires extending the field for $\sqrt{-3}$. Further, Algorithm~\ref{alg:peng_algorithm} can be generalized to any cubic $f(x)=a_3x^3+a_2x^2+a_1x+a_0$. We can first conjugate as in Equation \eqref{eq:cubiconjugate} and consider $h(x)=b_3x^3+b_1x+b_0.$ The algorithm can then be adjusted by the lead coefficient, $b_3,$ following the expressions in Proposition~\ref{prop:cubic}. 
\end{remark}

\bibliographystyle{plain}
\bibliography{RNT4}

\end{document}